\documentclass[a4paper, 11pt]{amsart}

\usepackage{epsfig}
\usepackage{amsthm,amsfonts}
\usepackage{amssymb,graphicx,color}
\usepackage[all]{xy}
\usepackage{cancel} 
\usepackage{verbatim}
\usepackage{hyperref}
\usepackage{enumitem}
\usepackage{neuralnetwork}
\usepackage{tikz}
\usepackage{tikz-cd}
\usepackage{graphicx}
\usepackage{float}

\usetikzlibrary{arrows}

\newtheorem{theorem}{Theorem}[section]
\newtheorem*{theorem*}{Theorem}
\newtheorem{lemma}[theorem]{Lemma}
\newtheorem{corollary}[theorem]{Corollary}
\newtheorem{proposition}[theorem]{Proposition}
\newtheorem{definition}[theorem]{Definition}
\newtheorem{example}[theorem]{Example}
\newtheorem{remark}[theorem]{Remark}
\newtheorem{claim}{Claim}[theorem]

\newcommand{\R}{\mathbb{R}}
\newcommand{\C}{\mathbb{C}}

\newcommand{\N}{\mathbb{N}}

\graphicspath{{Images/}}

\begin{document}

\title[Classification of algebraic curves under blow-spherical equiv.]
{Classification of complex algebraic curves under blow-spherical equivalence}

\author[J. E. Sampaio]{Jos\'e Edson Sampaio}
\author[E. C. da Silva]{Euripedes Carvalho da Silva}

\address{Jos\'e Edson Sampaio:  
              Departamento de Matem\'atica, Universidade Federal do Cear\'a,
	      Rua Campus do Pici, s/n, Bloco 914, Pici, 60440-900, 
	      Fortaleza-CE, Brazil. \newline  
              E-mail: {\tt edsonsampaio@mat.ufc.br}                    
}

\address{Euripedes Carvalho da Silva: Departamento de Matem\'atica, Instituto Federal de Educa\c{c}\~ao, Ci\^encia e Tecnologia do Cear\'a,
 	      Av. Parque Central, 1315, Distrito Industrial I, 61939-140, 
 	      Maracana\'u-CE, Brazil. \newline 
 	      and 	      Departamento de Matem\'atica, Universidade Federal do Cear\'a,
	      Rua Campus do Pici, s/n, Bloco 914, Pici, 60440-900, 
	      Fortaleza-CE, Brazil.
               E-mail: {\tt euripedes.carvalho@ifce.edu.br}
 } 

\thanks{The first named author was partially supported by CNPq-Brazil grant 310438/2021-7. This work was supported by the Serrapilheira Institute (grant number Serra -- R-2110-39576).
}

\begin{abstract}
This article is devoted to studying complex algebraic sets under (global) blow-spherical equivalence. This equivalence lives strictly between semialgebraic bi-Lipschitz equivalence  and topological equivalence.  
The main results of this article are complete classifications of complex algebraic curves. Firstly, we present a complete classification of complex algebraic curves under blow-spherical homeomorphisms at infinity and, then, we present a complete classification of complex algebraic curves under (global) blow-spherical homeomorphisms. For the classification at infinity we also present a classification with normal forms. We also present several properties of the blow-spherical equivalence. For instance, we prove that the degree of curves is preserved under blow-spherical homeomorphisms at infinity. Another property presented here is a Bernstein-type result which says that a pure dimensional complex algebraic set which is blow-spherical homeomorphic at infinity to a Euclidean space must be an affine linear subspace.

\end{abstract}

\maketitle

\tableofcontents

\section{Introduction}

One of the most natural and important problems in mathematics is the problem of classifying objects into a given category. For instance, an important result of classification is the classification of compact (without boundary) smooth surfaces under diffeomorphisms. It is well known the following: {\it Let $X\subset \R^n$ and $Y\subset \R^m$ be two connected smooth (without boundary) compact surfaces. Then the following statements are equivalent:
	\begin{itemize}
	 \item [(1)] $X$ and $Y$ are diffeomorphic;
	 \item [(2)] $X$ and $Y$ are (inner or outer) lipeomorphic (see the definition of lipeomorphism in Definition \ref{lipschitz function});
	 \item [(3)] $\theta(X)=\theta(Y)$ and $g(X)=g(Y)$,
	\end{itemize}}
\noindent where $g(M)$ denotes the genus of $M$, $\theta(M)=1$ if $M$ is orientable and $\theta(M)=-1$ if $M$ is non-orientable.

In the setting of (not necessarily compact) properly embedded smooth surfaces in $\R^n$, the problem of classifying is much harder and changes drastically, as it is shown in the next example. 

\begin{example}
	Let $X=\R^2$, $Y=\{(x,y,z)\in \R^3;x^2+y^2=1\}$ and $Z=\{(x,y,z)\in \R^3;z=x^2+y^2\}$. 
	\begin{itemize}
	 \item [a)] $\theta(X)=\theta(Y)$, $g(X)=g(Y)$, but $X$ and $Y$ are not homeomorphic;
	 \item [(b)] $X$ and $Z$ are diffeomorphic, but they are not inner lipeomorphic;
	\end{itemize}
\end{example} 

If we allow singularities then the problem of classifying is even harder.

In order to have some control on the topology of such surfaces, let us assume they are semialgebraic. Since compact manifolds (without boundary) are diffeomorphic to semialgebraic ones (see \cite{Nash}), in some sense that assumption is not too restrictive. With that assumption on the surfaces, the first author and Fernandes in \cite{FernandesS:2022} presented a complete classification of all semialgebraic surfaces with isolated singularities under inner lipeomorphisms.

Since there is no local classification of semialgebraic surfaces under outer lipeomorphisms, we believe that we are far away from obtaining a global classification of such surfaces under outer lipeomorphisms.
However, any semialgebraic outer lipeomorphism is a blow-spherical homeomorphism (see Definition \ref{def:bs_homeomorphism} and Proposition \ref{Lip_implies_bs}), thus an intermediate step in the problem of classifying complex algebraic surfaces or semialgebraic surfaces under outer lipeomorphisms is to classify algebraic surfaces under blow-spherical homeomorphisms. In order to learn more about the local properties of the blow-spherical equivalence see \cite{BirbrairFG:2012}, \cite{BirbrairFG:2017}, \cite{Sampaio:2015}, \cite{Sampaio:2020a}, \cite{Sampaio:2020} and \cite{Sampaio:2021b}. 

The problem of classification of the complex algebraic curves under outer lipeomorphisms was recently done by Targino in \cite{Targino:2022}. This is an important result alone, but it is also important to classify complex algebraic surfaces under outer lipeomorphisms, since an outer lipeomorphism between two complex algebraic surfaces induces an outer lipeomorphism between their singular sets (see \cite{Sampaio:2016} and \cite{BirbrairFLS:2016}), which are points or complex algebraic curves. Let us remark that there are algebraic curves which are blow-spherical homeomorphic, but are not outer lipeomorphic (see Example \ref{main_example}). But the blow-spherical equivalence is still a strong enough equivalence which is capable to catch singular points. In fact, it follows from \cite{Sampaio:2020} that a blow-spherical homeomorphism between two complex algebraic surfaces induces a blow-spherical homeomorphism between their singular sets.

So, it becomes natural to try classifying complex algebraic curves under blow-spherical homeomorphisms and the main aim of this article is to present such a classification (see Theorems \ref{thmequivinfinity} and \ref{main_theorem}). More precisely, we present two classifications: one under blow-spherical homeomorphisms at infinity (see Theorem \ref{thmequivinfinity}) and one under (global) blow-spherical homeomorphisms (see Theorem \ref{main_theorem}). Moreover, for the classification under blow-spherical homeomorphisms at infinity, we also present normal forms (see Theorem \ref{thm:normal_forms}).

It is important to say that the classifications presented here hold true for spacial algebraic curves, not only for plane algebraic curves as it was done in \cite{Targino:2022}. Moreover, since there are spacial algebraic curves which are not blow-spherical homeomorphic to any plane algebraic curve (see Remark \ref{rem:plane_spacial_are_not-same}), the problem of classifying spacial algebraic curves is harder than the problem of classifying plane algebraic curves.

In Section \ref{sec:app_lne}, as consequences of the results in \cite{Targino:2022} and \cite{FernandesS:2022}, we present some results related to LNE sets (see Definition \ref{def:lne}). For instance, a complex algebraic curves is LNE if and only if it is blow-spherical homeomorphic to a LNE complex algebraic curve (see Proposition \ref{prop:charac_bs_lne}); and two LNE complex algebraic curves blow-spherical homeomorphic if and only if they are outer lipeomorphic (see Corollary \ref{cor:bs_lipeo}).

\bigskip

\section{Preliminaries}\label{section:cones}
Here, we assume that all the algebraic sets are of pure dimension.
\subsection{Tangent Cones}
Let $X\subset \R^{n+1}$ be an unbounded semialgebraic set (resp. subanalytic set with $p\in \overline{X}$).
We say that $v\in \R^{n+1}$ is a tangent vector of $X$ at infinity (resp. $p$) if there are a sequence of points $\{x_i\}\subset X$ tending to infinity (resp. $p$) and a sequence of positive real numbers $\{t_i\}$ such that 
$$\lim\limits_{i\to \infty} \frac{1}{t_i}x_i= v \quad (\mbox{resp. } \lim\limits_{i\to \infty} \frac{1}{t_i}(x_i-p)= v).$$
Let $C(X,\infty)$ (resp. $C(X,p)$) denote the set of all tangent vectors of $X$ at infinity (resp. $p$). We call $C(X,\infty)$ the {\bf tangent cone of $X$ at infinity} (resp. $p$).

\begin{remark}
If $A\subset \C^n$ is a complex algebraic set and $p\in A\cup \{\infty\}$ then $C(A,p)$ is the zero set of a set of homogeneous polynomials (see \cite[p. 84, Proposition 2]{Chirka:1989}, \cite[Theorem 1.1]{LeP:2018} and \cite[Theorem 3.1]{Sampaio:2023}). In particular, $C(A,p)$ is a union of complex lines passing through at the origin.
\end{remark}

We have the following characterization.
\begin{corollary}[Corollary 2.16 \cite{FernandesS:2020}]\label{corollary 2.16FernandesS:2020}
Let $X\subset \R^n$ be an unbounded semialgebraic set. Then  
$C(X,\infty)=\{v\in\R^n;\, \exists \gamma:(\varepsilon,+\infty )\to Z$ $C^0$ semialgebraic such that $\lim\limits _{t\to +\infty }|\gamma(t)|=+\infty$ and $\gamma(t)=tv+o_{\infty }(t)\}$, where $g(t)=o_{\infty }(t)$ means $\lim\limits_{t\to +\infty}\frac{g(t)}{t}=0$.
\end{corollary}
%

Thus, we have the following 
\begin{corollary}[Corollary 2.18 \cite{FernandesS:2020}]\label{dimension_tg_cone}
Let $Z\subset \R^n$ be an unbounded semialgebraic set. Let $\phi:\R^n\setminus\{0\}\to \R^n\setminus\{0\}$ be the semialgebraic mapping given by $\phi(x)=\frac{x}{\|x\|^2}$ and denote $X=\phi(Z\setminus \{0\})$. Then $C(Z,\infty)$ is a semialgebraic set satisfying $C(Z,\infty)=C(X,0)$ and $\dim_{\R} C(Z,\infty) \leq \dim_{\R} Z$.
\end{corollary}

Another way to present the tangent cone at infinity (resp. $p$) of a subset $X\subset\R^{n+1}$ is via the spherical blow-up at infinity (resp. $p$) of $\R^{n+1}$. Let us consider the {\bf spherical blowing-up at infinity} (resp. $p$) of $\R^{n+1}$, $\rho_{\infty}\colon\mathbb{S}^n\times (0,+\infty )\to \R^{n+1}$ (resp. $\rho_p\colon\mathbb{S}^n\times [0,+\infty )\to \R^{n+1}$), given by $\rho_{\infty}(x,r)=\frac{1}{r}x$ (resp. $\rho_p(x,r)=rx+p$).

Note that $\rho_{\infty}\colon\mathbb{S}^n\times (0,+\infty )\to \R^{n+1}\setminus \{0\}$ (resp. $\rho_p\colon\mathbb{S}^n\times (0,+\infty )\to \R^{n+1}\setminus \{0\}$) is a homeomorphism with inverse mapping $\rho_{\infty}^{-1}\colon\R^{n+1}\setminus \{0\}\to \mathbb{S}^n\times (0,+\infty )$ (resp. $\rho_p \colon\mathbb{S}^n\times (0,+\infty )\to \R^{n+1}\setminus \{0\}$) given by $\rho_{\infty}^{-1}(x)=(\frac{x}{\|x\|},\frac{1}{\|x\|})$ (resp. $\rho_p^{-1}(x)=(\frac{x-p}{\|x-p\|},\|x-p\|)$).

 The {\bf strict transform} of the subset $X$ under the spherical blowing-up $\rho_{\infty}$ is $X'_{\infty}:=\overline{\rho_{\infty}^{-1}(X\setminus \{0\})}$ (resp. $X'_{p}:=\overline{\rho_{p}^{-1}(X\setminus \{0\})}$). The subset $X_{\infty}'\cap (\mathbb{S}^n\times \{0\})$ (resp. $X_{p}'\cap (\mathbb{S}^n\times \{0\})$) is called the {\bf boundary} of $X'_{\infty}$ (resp. $X'_p$) and it is denoted by $\partial X'_{\infty}$ (resp. $\partial X'_p$). 

\begin{remark}\label{remarksimplepointcone}
{\rm If $X\subset \R^{n+1}$ is a semialgebraic set, then $\partial X'_{\infty}=(C(X,\infty)\cap \mathbb{S}^n)\times \{0\}$ (resp. $\partial X'_{p}=(C(X,p)\cap \mathbb{S}^n)\times \{0\}$).}
\end{remark}

\subsection{Relative multiplicities}

Let $X\subset \R^{m+1}$ be a $d$-dimensional subanalytic subset and $p\in \R^{m+1}\cup \{\infty\}$. We say $x\in\partial X'_p$ is {\bf a simple point of $\partial X'_p$}, if there is an open subset $U\subset \R^{m+2}$ with $x\in U$ such that:
\begin{itemize}
\item [a)] the connected components $X_1,\cdots , X_r$ of $(X'_p\cap U)\setminus \partial X'_{p}$ are topological submanifolds of $\R^{m+2}$ with $\dim X_i=\dim X$, for all $i=1,\cdots,r$;
\item [b)] $(X_i\cup \partial X'_{p})\cap U$ are topological manifolds with boundary, for all $i=1,\cdots ,r$. 
\end{itemize}
Let ${\rm Smp}(\partial X'_{p})$ be the set of simple points of $\partial X'_{p}$ and we define $C_{\rm Smp}(X,p)=\{t\cdot x; \, t>0\mbox{ and }x\in {\rm Smp}(\partial X'_{p})\}$. Let $k_{X,p}\colon {\rm Smp}(\partial X'_{p})\to \N$ be the function such that $k_{X,p}(x)$ is the number of connected components of the germ $(\rho_{p}^{-1}(X\setminus\{p\}),x)$. 

\begin{remark}\label{remarksimplepointdense}
	${\rm Smp}(\partial X'_p)$ is an open dense subset of the $(d-1)$-dimensional part of $\partial X'_p$ whenever $\partial X'_p$ is a $(d-1)$-dimensional subset, where $d=\dim X$. 
\end{remark} 
\begin{definition}\label{def:relative_mult}
It is clear the function $k_{X,p}$ is locally constant. In fact, $k_{X,p}$ is constant on each connected component $C_j$ of ${\rm Smp}(\partial X'_{p})$. Then, we define {\bf the relative multiplicity of $X$ at $p$ (along of $C_j$)} to be $k_{X,p}(C_j):=k_{X,p}(x)$ with $x\in C_j$. 
Moreover, when $X$ is a complex algebraic set, there is a complex algebraic set $\sigma$ with $\dim \sigma <\dim X$, such that $X_j\setminus \sigma$ intersect only one connected component $C_i$ (see \cite{Chirka:1989}, pp. 132-133), for each irreducible component $X_j$ of tangent cone $C(X,p)$. Then we define also $k_{X,p}(X_j):=k_{X,p}(C_i)$.
Let $X_1,...,X_r$ be the irreducible components of $C(X,p)$. By reordering indices, if necessary, we assume that $k_{X,p}(X_1)\leq \cdots \leq k_{X,p}(X_r)$. Then we define $k(X,p)=(k_{X,p}(X_1),...,k_{X,p}(X_r))$.
\end{definition}

%

%

\subsection{Degree of complex algebraic sets}

\begin{definition}
Let $X\subset \mathbb{C}^n$ be a pure $p$-dimensional complex algebraic set such that and
let $L\in G(n-p,n)$ such that $L\cap C(Y ,\infty) = \{0\}$. Let $\Pi \colon \mathbb{C}^n\rightarrow \mathbb{C}^p$ be the orthogonal projection such that $L=\Pi^{-1}(0)$. Therefore, there exists a proper algebraic subset $\sigma\subset \mathbb{C}^p$ such that $d=\#(\Pi^{-1}(t)\cap X)$ does not depend on $t\in \C^p\setminus \sigma$. We highlight that, $d$ does not depend also on $L$. Thus, we define {\bf the degree of $X$} to be ${\rm deg}(X)=d$.
\end{definition}

\begin{proposition}[Proposition 3.3 in \cite{FernandesS:2020}]\label{propdegone}
Let $X\subset \mathbb{C}^n$ be a pure dimensional algebraic subset. Then, $\deg(X)=1$ if and only if $X$ is an affine linear subspace of $\mathbb{C}^n$.
\end{proposition}

\begin{remark}[\cite{BobadillaFS:2018}]\label{remarkdegree}
 Let $X\subset \mathbb{C}^n$ be a pure dimensional complex algebraic set and  $X_1, \cdots, X_r$ be the irreducible components of $C(X,\infty)$. Then 
 $$\deg(X)=\sum_{j=1}^r{k_{X,\infty}(X_j)\cdot \deg(X_j)}.$$
\end{remark}

\subsection{Ends of semialgebraic sets}
 
For $R>0$ we denote $\mathbb{S}^{n-1}_{R}=\{x\in \mathbb{R}^n; \|x\|=R\}$ and $B^{n}_R=\{x\in \mathbb{R}^n; \|x\|< R\}$. Moreover, if $X\subset \mathbb{R}^n$, we define $\mbox{Cone}_{\infty}(X)=\{tx; x\in X \ \ \mbox{and} \ \ t\in [1,+\infty)\}$.
\begin{proposition}\label{propconicalstru}
 Let $C\subset \mathbb{R}^n$ be a semialgebraic set. Then there exists $R_0>0$ such that for each $R\geq R_0$ there exists a homeomorphism $h\colon C\setminus B_R(0)\rightarrow \mbox{Cone}_{\infty}(C\cap \mathbb{S}^{n-1}_{R})$ such that $\|h(x)\|=\|x\|$ for all $x\in C\setminus B_R(0)$. Moreover, if $C$ has isolated singularities then we can choose $R_0$ such that $C\cap \mathbb{S}^{n-1}_R$ is smooth for all $R\geq R_0$.
\end{proposition}

Each connected component of $C\setminus B_R(0)$ is called an {\bf end of $C$}.

\section{The blow-spherical equivalence and some examples and properties}

\subsection{Definition of the blow-spherical equivalence}

\begin{definition}\label{def:bs_homeomorphism}
Let $X$ and $Y$ be subsets in $\mathbb{R}^n$ and $\mathbb{R}^m$ respectively. Let $p\in \mathbb{R}^n\cup \{\infty\}$, $q\in \mathbb{R}^m\cup \{\infty\}$. A homeomorphism $\varphi:X\rightarrow Y$ such that $q=\lim\limits_{x\rightarrow p}{\varphi(x)}$ is said a {\bf blow-spherical homeomorphism at $p$}, if  the homeomorphism 
$$\rho^{-1}_{q}\circ \varphi\circ \rho_p \colon X'_p\setminus \partial X'_p\rightarrow Y'_q\setminus \partial Y'_q$$
extends to a homeomorphism $\varphi'\colon X'_p\rightarrow Y'_q$. A homeomorphism $\varphi\colon X\rightarrow Y$ is said a {\bf blow-spherical homeomorphism} if it is a blow-spherical homeomorphism for all $p\in \overline{X}\cup \{\infty\}$. In this case, we say that the sets $X$ and $Y$ are {\bf blow-spherical homeomorphic} or {\bf blow-isomorphic} ({\bf at $(p,q)$}). 
\end{definition}

\begin{definition}\label{def:strong_diffeo}
Let $X$ and $Y$ be subsets in $\mathbb{R}^n$ and $\mathbb{R}^m$, respectively. We say that a blow-spherical homeomorphism $h\colon X\rightarrow Y$ is a {\bf strong blow-spherical homeomorphism} if $h({\rm Sing}_1(X))={\rm Sing}_1(Y)$ and $h|_{X\setminus {\rm Sing}_1(X)}\colon X\setminus {\rm Sing}_1(X)\rightarrow Y\setminus {\rm Sing}_1(Y)$ is a $C^1$ diffeomorphism, where, for $A\subset \R^p$, ${\rm Sing}_k(A)$ denotes the points $x\in A$ such that, for any open neighbourhood $U$ of $x$, $A\cap U$ is not a $C^k$ submanifold of $\R^p$. A blow-spherical homeomorphism at $\infty$, $\varphi\colon X\rightarrow Y$, is said a {\bf strong blow-spherical homeomorphism at $\infty$ } if there are compact sets $K\subset \mathbb{R}^n$ and $\tilde{K}\subset\mathbb{R}^m$ such that $\varphi(X\setminus K)=Y\setminus \tilde{K}$ and the restriction $\varphi|_{X\setminus K}\colon X\setminus K\to Y\setminus \tilde{K}$ is a strong blow-spherical homeomorphism. 
\end{definition}

\begin{remark}
 We have some examples:
 \begin{enumerate}
  \item $\mbox{Id}:X \rightarrow X$ is a blow-spherical homeomorphism for any $X\subset \mathbb{R}^n$;
  \item Let $X\subset \mathbb{R}^n$, $Y\subset \mathbb{R}^m$ and $Z\subset \mathbb{R}^k$ be subsets. If $f\colon X\rightarrow Y$ and $g\colon Y\rightarrow Z$ are blow-spherical homeomorphisms, then $g\circ f\colon X\rightarrow Z$ is a blow-spherical homeomorphism.
 \end{enumerate}

\end{remark}

Thus we have a category called {\bf blow-spherical category}, which is denoted by $\mbox{BS}$, where its objects are all the subsets of Euclidean spaces and its morphisms are all blow-spherical homeomorphisms.

By definition, if $X$ and $Y$ are strongly blow-spherical homeomorphic then they are blow-spherical homeomorphic, but the converse does not hold in general, as we can see in the next example.

\begin{example}
Let $V=\{(x,y)\in\R^2; y=|x|\}$. The mapping $\varphi\colon \R\to V$ given by $\varphi(x)=(x,|x|)$ is a blow-spherical homeomorphism (see Proposition \ref{Lip_implies_bs}). However, since ${\rm Sing}_1(\R)=\emptyset$ and ${\rm Sing}_1(V)=\{(0,0)\}$, there is no strong blow-spherical homeomorphism $h\colon \R\to V$ and, in particular, $\R$ and $V$ are not strongly blow-spherical homeomorphic.
\end{example}

\subsection{Blow-spherical invariance of the relative multiplicities}

\begin{proposition}\label{propinvarianciamultiplicidaderelativa}
	Let $X$ and $Y$ be semialgebraic set in $\mathbb{R}^n$ and $\mathbb{R}^m$ respectively. Let $\varphi:X\rightarrow Y$ be a blow-spherical homeomorphism at $p\in \mathbb{R}^n\cup \{\infty\}$. Then  
	$$k_{X,p}(x)=k_{Y,q}(\varphi'(x)),$$
	for all $x\in Smp(\partial X'_p)\subset \partial X'_p$, where $q=\lim\limits_{x\to p}\varphi(x)$. In particular, $k(X,p)=k(Y,q)$.
\end{proposition}
\begin{proof}
Since $\varphi\colon X\rightarrow Y$ is a blow-spherical homeomorphism at $p$, we have that $\varphi':X'_p\rightarrow Y'_q$ is a homeomorphism such that $\varphi'(\partial X'_p)=\partial Y'_q$ and $\varphi'({\rm Smp}(\partial X'_p))= {\rm Smp}(\partial Y'_q)$. Thus, for each $x\in {\rm Smp}(\partial X'_p)$, the number of connected components of the germ $(\rho_{p}^{-1}(X\setminus\{p\}),x)=(X'_p\setminus \partial X'_p,x)$ is equal to the number of connected components of the germ $(\rho_{q}^{-1}(Y\setminus\{q\}),\varphi'(x))=(Y'_q\setminus \partial Y'_q,\varphi'(x)) $, which implies  $k_{X,p}(x)=k_{Y,q}(\varphi'(x))$. 
\end{proof}

\subsection{Blow-spherical invariance of the tangent cones}
The next result is a generalization of Proposition 3.3 in \cite{Sampaio:2020}.

\begin{proposition}\label{coneinvarianteblow}
 If $h\colon X\rightarrow Y$ is a blow-spherical homeomorphism at $p\in X\cup \{\infty\}$, then $C(X,p)$ and $C(Y,h(p))$ are blow-spherical homeomorphic at $0$.
\end{proposition}

\begin{proof}
Let $q=\lim_{x\rightarrow p} h(x)$. Thus $h'_{p}\colon \partial X'_{p}\rightarrow \partial Y'_{q}$ is a homeomorphism. Let $\nu_{h}\colon C(X,p)\cap \mathbb{S}^{n-1}\rightarrow C(Y,q)\cap \mathbb{S}^{m-1}$ be the homeomorphism satisfying $h'_{p}(x,0)=(\nu_{h}(x),0)$ for all $x\in C(X,p)\cap \mathbb{S}^{n-1}$. Now, we define $d_{p}h\colon C(X,p)\rightarrow C(Y,q)$ by
  \[d_{p}h(x) = \left\{
  \begin{array}{ll}
    \|x\|\nu_h\left(\frac{x}{\|x\|}\right) &  \mbox{if} \ x \neq 0\\
    0 &  \mbox{if} \ x= 0.
  \end{array}
\right.
\] 
Since $\rho_{0}^{-1}\circ d_{p}h\circ \rho_{0}(x,t)= (\nu_h(x),t)$, $d_{p}h$ is a blow-spherical homeomorphism at $0$.
\end{proof}

\subsection{Blow-spherical invariance of the degree of curves}
\begin{proposition}\label{invariance_of_degree}
Let $X$ and $Y$ be two complex algebraic curves. 
If $X$ and $Y$ are blow-spherical homeomorphic at $\infty$, then ${\rm deg}(X)={\rm deg}(Y)$.
\end{proposition}
\begin{proof}
Let $X_1,...,X_r$ (resp. $Y_1,...,Y_s$) be the irreducible components of $C(X,\infty)$ (resp. $C(Y,\infty)$).
By Proposition \ref{coneinvarianteblow}, we have that $C(X,\infty)$ and $C(Y,\infty)$ are blow-spherical homeomorphic and, consequently, $r=s$. Since each $X_i$ (resp. $Y_j$) is a complex line, it follows from Proposition \ref{propinvarianciamultiplicidaderelativa} and Remark \ref{remarkdegree} that 
 \begin{eqnarray*}
 	\deg X & = & \sum_{i=1}^r{k_{X,\infty}(X_i)\deg X_i}\\
 	& = & \sum_{i=1}^r{k_{X,\infty}(X_i)}\\
 	& = & \sum_{i=1}^r{k_{Y,\infty}(Y_i)}\\
 	& = & \sum_{i=1}^r{k_{Y,\infty}(Y_i)\deg Y_i}=\deg Y.
 \end{eqnarray*}
\end{proof}

\subsection{A Bernstein-type result}

\begin{definition}
	A subset $X\subset \mathbb{R}^n$ is called blow-spherical regular at infinity if there are compact subsets $K$ and $\tilde{K}$ in $X$ and $\mathbb{R}^d$ respectively such that $X\setminus K$ is a blow-spherical homeomorphic at $\infty$ to an $\mathbb{R}^d\setminus \tilde{K}$.
\end{definition}
Now we remember Prill's Theorem proved in \cite{Prill:1967}.
\begin{lemma}[\cite{Prill:1967}, Theorem]
	Let $C\subset \mathbb{C}^n$ be a complex cone which is a topological manifold. Then $C$ is a linear subspace of $\mathbb{C}^n$.
\end{lemma}
\begin{proposition}
 Let $X\subset \mathbb{C}^n$ be a pure dimensional complex algebraic set. If $X$ is blow-spherical regular at infinity, then $X$ is an affine linear subspace of $\mathbb{C}^n$.
\end{proposition}
\begin{proof}
 By Proposition \ref{coneinvarianteblow}. $C(X,\infty)$ is homeomorphic to $\mathbb{C}^k$, where $k=\dim X$, consequently, $C(X,\infty)$ is a topological manifold. By Prill's theorem $C(X,\infty)$ is a linear subspace of $\mathbb{C}^n$, in particular $C(X,\infty)$ is irreducible. Hence $\deg(C(X,\infty))=1$ and by Proposition \ref{propinvarianciamultiplicidaderelativa}, we have $k_{X,\infty}(C(X,\infty))=1$ and using Remark \ref{remarkdegree}, we have that
 $$\deg X=k_{X,\infty}(C(X,\infty))\cdot \deg(C(X,\infty))=1.$$
 Then, by Proposition \ref{propdegone}, $X$ is an affine linear subspace of $\mathbb{C}^n$.
\end{proof}

\subsection{Semi-algebraic outer lipeomorphisms are blow-spherical homeomorphisms}

Given a path connected subset $X\subset\R^n$, the
{\bf inner distance}  on $X$  is defined as follows: given two points $x_1,x_2\in X$, $d_{X,inn}(x_1,x_2)$  is the infimum of the lengths of paths on $X$ connecting $x_1$ to $x_2$. We denote by $d_{X,out}$ the Euclidean distance of $\R^n$ restrict to $X$.
\begin{definition}\label{lipschitz function}
Let $X\subset\R^n$ and $Y\subset\R^m$. A mapping $f\colon X\rightarrow Y$ is called {\bf outer} (resp. {\bf inner}) {\bf Lipschitz} if there exists $\lambda >0$ such that is
$$\|f(x_1)-f(x_2)\|\le \lambda \|x_1-x_2\| \quad (\mbox{resp. } d_{X,inn}(f(x_1),f(x_2))\le \lambda d_{X,inn}(x_1,x_2))$$  for all $x_1,x_2\in X$. An outer Lipschitz (resp. inner Lipschitz) mapping $f\colon X\rightarrow Y$ is called
{\bf outer} (resp. {\bf inner}) {\bf lipeomorphism} if its inverse mapping exists and is outer Lipschitz (resp. inner Lipschitz) and, in this case, we say that  $X$ and $Y$ are {\bf outer} (resp. {\bf inner}) {\bf lipeomorphic}.
\end{definition}
\begin{proposition}\label{Lip_implies_bs}
Let $X\subset \mathbb{R}^n$ and $Y\subset \mathbb{R}^m$ be semialgebraic sets. If $h\colon X\rightarrow Y$ is a semialgebraic outer lipeomorphism then $h$ is a blow-spherical homeomorphism. 
\end{proposition}
\begin{proof}
Firstly, let us prove that $h$ is a blow-spherical homeomorphism at $\infty$.
 So, consider $h_{\infty}'\colon X_{\infty}'\rightarrow Y_{\infty}'$ given by
 \[h_{\infty}'(u,t) = \left\{
  \begin{array}{lr}
    \left( \frac{h(\frac{u}{t})}{\|h(\frac{u}{t})\|},\frac{1}{\|h(\frac{u}{t})\|}\right) & \mbox{if} \ t \neq 0\\
    \left( \frac{d_{\infty}h(u)}{\|d_{\infty}h(u)\|},0\right) & \mbox{if} \  t= 0
  \end{array}
\right.
\]
where $d_{\infty}h\colon C(X,\infty)\rightarrow C(Y,\infty)$ is the mapping given by in the following way: Given $v\in C(X,\infty)$, let $\gamma: (\epsilon, +\infty)\rightarrow X$ be a semialgebraic arc such that $\gamma(t)=tv+o_{\infty}(t)$. Then $h\circ \gamma(t)=tw+o_{\infty}(t)$, for some $w\in C(Y,\infty)$. Thus, we set $d_{\infty}h(v)=w$.
\begin{claim}\label{caim:dh-bilipischitz}
$d_{\infty}h\colon C(X,\infty)\rightarrow C(Y,\infty)$ is an outer lipeomorphism.
\end{claim}
\begin{proof}[Proof of Claim \ref{caim:dh-bilipischitz}] 
By hypotheses, there is a $K\geq 1$ such that 
$$\frac{1}{K}\|x-y\| \leq \|h(x)-h(y)\|\leq K\|x-y\|.$$
Now, consider $v_1,v_2\in C(X,\infty)$ and let $\gamma_1,\gamma_2 \colon (\epsilon,+\infty)\rightarrow X$ semialgebraic curve such that $\gamma_i(t)=tv_i+o_{\infty}(t)$, for $i=1,2$. Therefore,
$$\frac{1}{K}\frac{\|\gamma_1(t)-\gamma_2(t)\|}{t}\leq \frac{\|h(\gamma_1(t))-h(\gamma_2(t))\|}{t}\leq K\frac{\|\gamma_1(t)-\gamma_2(t)\|}{t}.$$
Sending $t$ to infinity, we obtain 
$$\frac{1}{K}\|v_1-v_2\|\leq \|d_{\infty}h(v_1)-d_{\infty}h(v_2)\| \leq K\|v_1-v_2\|.$$ 
Thus, $d_{\infty}h$ is well-defined and an outer lipeomorphism. 
\end{proof}
Therefore, 
$$\frac{d_{\infty}h}{\|d_{\infty}h\|}\colon C(X,\infty)\cap \mathbb{S}^n\rightarrow C(Y,\infty)\cap \mathbb{S}^n$$ is a homeomorphism. 
\begin{claim}\label{claim:h-continuous}
$h_{\infty}' \colon X_{\infty}' \rightarrow Y_{\infty}'$ is a continuous map.
\end{claim}
\begin{proof}[Proof of Claim \ref{claim:h-continuous}.]
Clearly, $h_{\infty}'$ is continuous in $X_{\infty}'\setminus \partial X_{\infty}'$. Thus, it is enough to prove that $h_{\infty}'$ is continuous at each point $(u,0)\in \partial X_{\infty}'$. So, let $\{(x_n,t_n)\}\subset X_{\infty}'$ such that $(x_n,t_n)\rightarrow (u,0)$. We have to show that 
$$\lim_{n\rightarrow +\infty} h_{\infty}'(x_n,t_n)=h_{\infty}'(u,0)=\left(\frac{d_{\infty}h(u)}{\|d_{\infty}h(u)\|},0\right).$$
Since $h_{\infty}'{\big|}_{\partial X_{\infty}'}\colon \partial X_{\infty}'\rightarrow \partial Y_{\infty}'$  is continuous, we may assume that $t_n> 0$, $\forall n$. Thus, by setting $s_n=\frac{1}{t_n}$, we have 
\begin{eqnarray*}
 \lim_{n\rightarrow +\infty} h_{\infty}'(x_n,t_n) & = & \lim_{n\rightarrow +\infty} \left( \frac{h(\frac{u}{t_n})}{\|h(\frac{u}{t_n})\|},\frac{1}{\|h(\frac{u}{t_n})\|}\right)\\
 & = & \lim_{n\rightarrow +\infty} \left( \frac{\frac{h(s_nu)}{s_n}}{\|\frac{h(s_nu)}{s_n}\|},\frac{1}{\|h(s_nu)\|}\right).
\end{eqnarray*}
Let $\gamma\colon (\epsilon,+\infty)\rightarrow X$ be a semialgebraic arc such that $\gamma(s)=su+o_{\infty}(s)$. Then 
$\lim_{n\rightarrow +\infty}\frac{h(s_nu)}{s_n}= d_{\infty}h(u)$.
Indeed, 
\begin{eqnarray*}
 \left\| \frac{h(s_nu)}{s_n}-\frac{h(\gamma(s_n))}{s_n}\right\| & \leq & K\frac{\|s_nu-\gamma(s_n)\|}{s_n}\\
 & = & K\left\| \frac{o_{\infty}(s_n)}{s_n}\right\|.
\end{eqnarray*}
Therefore, 
\begin{eqnarray*}
 0 & = & \lim_{n\rightarrow +\infty}\left\| \frac{h(s_nu)}{s_n}-\frac{h(\gamma(s_n))}{s_n}\right\|\\
 & = & \left\| \lim_{n\rightarrow +\infty} \frac{h(s_nu)}{s_n}-d_{\infty}h(u)\right\|.
\end{eqnarray*}
which implies that $h'_{\infty}$ is continuous.
\end{proof}

Similarly, we also prove that $h'^{-1}_{\infty}$ is continuous and this finishes the proof that $h$ is a blow-spherical homeomorphism at $\infty$.

Now, let $p\in \overline{X}$ and $q=\lim\limits_{x\to p}h(p)$. 
Let $T_p\colon \R^n\to \R^n$ and $T_q\colon \R^m\to \R^m$ the mappings given by $T_p(x)=x+p$ and $T_q(y)=x+q$. 
It follows from \cite[Proposition 3.8]{Sampaio:2020} that $T_q^{-1}\circ h\circ T_p$ is a blow spherical homeomorphism at $0$. Therefore, $h$ is a blow spherical homeomorphism at $p$.

\end{proof}

\subsection{The projective closure of algebraic curves which are blow-spherical homeomorphic}
Let $\iota\colon\mathbb{C}^n\rightarrow \mathbb{CP}^n$ be the mapping given by $\iota(x)=(1:x)$. The hyperplane at infinity, denoted by $L_{\infty}$, is the complement of $\iota(\mathbb{C}^n)$ in $\mathbb{CP}^n$. Here, we identify $\C^n$ with $\iota(\mathbb{C}^n)$. With this identification, for a set $X\subset \C^n$, the closure of $X$ in $\mathbb{CP}^n$ is called the {\bf projective closure of $X$}.
\begin{proposition}\label{homeoprojective}
 Let $X\subset \mathbb{C}^m$ and $Y\subset \mathbb{C}^k$ be complex algebraic curves. Let $h\colon X\rightarrow Y$ be a homeomorphism which is a blow-spherical homeomorphism at $\infty$. Then $h$ extends to a homeomorphism between the projective closures of $X$ and $Y$.
\end{proposition}
\begin{proof}
 Let $\overline{X}$ and $\overline{Y}$ be the projective closures of $X$ and $Y$, respectively. Thus $$h\colon \overline{X}\setminus L_{\infty}\rightarrow \overline{Y}\setminus L_{\infty}$$ 
 
 $$(1:z)\mapsto (1:h(z))$$
 Let $\{a_1,\cdots, a_r\}=\overline{X}\cap L_{\infty}$ and $\{b_1,\cdots, b_s\}=\overline{Y}\cap L_{\infty}$.

We write $a_i=(0:c_i)$ (resp. $b_j=(0:d_j)$) for $i=1,\cdots, r$ (resp. $j=1,\cdots, s$). Note that $C(X, \infty)=\bigcup{\mathbb{C}\cdot c_i}$ (resp. $C(Y, \infty)=\bigcup{\mathbb{C}\cdot d_j}$). Since $h$ is a blow-spherical homeomorphism at $\infty$, then $h$ induces a homeomorphism $\phi\colon C(X,\infty)\rightarrow C(Y,\infty)$ and, consequently, there is a bijection $\sigma\colon \{1,\cdots, r\}\rightarrow \{1,\cdots, s\}$ such that
$$\phi(\mathbb{C}\cdot c_i)=\mathbb{C}\cdot d_{\sigma(i)},$$
where $\phi\colon C(X,\infty)\rightarrow C(Y,\infty)$ is defined as follows
 \[\phi(v) = \left\{
  \begin{array}{lr}
    \|v\|\cdot \nu_{\infty}\left(\frac{v}{\|v\|} \right), &  v \neq 0\\
    0, &  v= 0.
  \end{array}
\right.
\]
and $\nu_{\infty}\colon C(X,\infty)\cap \mathbb{S}^{2m-1}\rightarrow C(Y,\infty)\cap \mathbb{S}^{2k-1} $ is the homeomorphism such that $h_{\infty}'(v,0)=\left(\nu_{\infty}(v),0\right)$.

Thus, we define $\overline{h}\colon \overline{X}\rightarrow \overline{Y}$ by
 \[\overline{h}(v) = \left\{
  \begin{array}{ll}
    h(v), &  v \in \overline{X}\setminus L_{\infty}\\
    b_{\sigma(i)}, &  v= a_i
  \end{array}
\right.
\]
Fix $i\in \{1,\cdots,r\}$ and let $(z_n:x_n)\in \overline{X}$ with $z_n\neq 0$ such that $(z_n:x_n)\rightarrow a_i=(0:c_i)$.   

Since $c_i=(c_i^1:...:c_i^m)\not=0$, we may assume, without loss of generality, $c_i^1\neq 0$. So, there exists $n_0\in \mathbb{N}$ such that $x_n^1\neq 0$ $\forall n\geq n_0$, where $x_n=(x_n^1:...:x_n^m)=(x_n^1:y_n)$. Then
$$w_n=\left( 1:\frac{x_n^1}{z_n}:\frac{y_n}{z_n}\right)\rightarrow (0:1:c).$$

Note that if $\{u_n\}\subset X$ such that $\frac{u_n}{\|u_n\|}\rightarrow u$ then there exists $n_0\in \mathbb{N}$ such that $u_n\in \bigcup_{i\in J}{X_i}$, $\forall n>n_0$ and $C(X_i,\infty)=L=\mathbb{C}\cdot u$, $\forall i\in J$. Moreover, $h(u_n)\in \bigcup_{j\in J'}{Y_j}$ such that $C(Y_j,\infty)=\phi(L)=L'$, $\forall j\in J'$. In fact, 
$$h_{\infty}'\left(\frac{u_n}{\|u_n\|},\frac{1}{\|u_n\|}\right)= \left(\frac{h(u_n)}{\|h(u_n)\|},\frac{1}{\|h(u_n)\|}\right)\rightarrow h_{\infty}'(u,0)$$
and, in particular, $\frac{h(u_n)}{\|h(u_n)\|}\rightarrow \phi(u)\in \phi(L)=L'=\mathbb{C}\cdot v$. 
This implies that 
\begin{eqnarray*}
\overline{h}(w_n) & = & \overline{h}\left(1:\frac{x_n^1}{z_n}:\frac{y_n}{z_n}\right)=\left(1:h\left(\frac{x_n^1}{z_n},\frac{y_n}{z_n}\right)\right) \rightarrow b_{\sigma(i)},
\end{eqnarray*} 
which shows that $\overline{h}$ is continuous. Similarly, we prove that $\overline{h}^{-1}$ is continuous as well.
Therefore, $\overline{h}:\overline{X}\rightarrow \overline{Y}$ is a homeomorphism.
\end{proof}

The next example shows that we can not remove the condition that $h$ is a blow-spherical homeomorphism at $\infty$ in Proposition \ref{homeoprojective}.

\begin{example}
Let $X=\{(x,y)\in \C^2;xy=0\}$ and $Y=\{(x,y)\in \C^2;(y-x^2)x=0\}$. Then $X$ and $Y$ are homeomorphic, but their projective closures $\overline{X}$ and $\overline{X}$ are not homeomorphic. Indeed, $(\overline{X},p)$ is irreducible for any $p\in \overline{X}\setminus \{(1:0:0)\}$ and $(\overline{Y},p)$ is reducible for any $p\in \{(1:0:0),(0:0:1)\}$ (see also Figures \ref{fig:complexbstreeinfinity_one} and \ref{fig:complexbstreeinfinity_two}). 
\end{example}

As we can see in the next example, the reciprocal of Proposition \ref{homeoprojective} is not true.
\begin{example}
Let $L=\{(x,y)\in \C^2;y=0\}$ and $C=\{(x,y)\in \C^2;x^2+y^2=1\}$. Then the projective closures $\overline{L}$ and $\overline{C}$ are smooth projective curves with genus zero. Therefore, $\overline{L}$ and $\overline{C}$ are diffeomorphic. However, $L$ and $C$ are not blow-spherical homeomorphic at infinity and, in particular, they are not blow-spherical homeomorphic (see also Figures \ref{fig:complexbstreeinfinity_three} and \ref{fig:complexbstreeinfinity_four}).
\end{example}


\section{Classification of complex algebraic curves under blow-spherical equivalence}

\subsection{Classification of complex algebraic curves at infinity}
We are going to prove that blow-spherical geometry at infinity of a complex algebraic curve determines and is determined by the blow-spherical tree at infinity.

\begin{definition}
Let $C\subset \mathbb{C}^n$ be a complex algebraic curve and let $X_1,\cdots , X_r$ be the irreducible components of the tangent cone at infinity $C(C,\infty)$.  The {\bf complex blow-spherical tree at infinity of $C$} is the rooted tree with a root with the label $\infty$, with vertices $D_i$'s corresponding to the lines $X_i$'s and we put edges joining each vertex $D_i$ to the root. Finally, for each direction at infinity $X_i$, we put, for each end $E_{i,j}$ of the curve $C$ which is tangent to the $X_{i}$, a new vertex $A_{i,j}$ and an edge joining $A_{i,j}$ to $D_i$ with weight $k_{E_{i,j}}(X_i)$.
\end{definition}

Some examples of complex blow-spherical trees at infinity are presented in Figures \ref{fig:complexbstreeinfinity_one},  \ref{fig:complexbstreeinfinity_two},  \ref{fig:complexbstreeinfinity_three} and  \ref{fig:complexbstreeinfinity_four}.

\begin{figure}[H]
\begin{tabular}{cc}
\begin{minipage}[c][6cm][c]{6cm}
	\centering
	\begin{tikzpicture}[line cap=round,line join=round,>=triangle 45,x=0.7cm,y=0.7cm]
		\clip(-5,-1) rectangle (5,6.2);
		\draw [line width=1pt] (0,0)-- (0,2);
		\draw [line width=1pt] (0,2)-- (1,4);
		\draw [line width=1pt] (0,2)-- (-1,4);
		\draw [line width=1pt] (0,0) circle (2pt);
		\begin{scriptsize}
			\draw [fill=black] (0,0) circle (2pt);
			\draw [fill=black] (0,2) circle (2pt);
			\draw [fill=black] (-1,4) circle (2pt);
			\draw [fill=black] (1,4) circle (2pt);
			\draw [fill=black] (1 ,3) node {\normalsize $1$};
			\draw [fill=black] (-1 ,3) node {\normalsize $2$};
			\draw [fill=black] (0.4,0) node {\normalsize $\infty$};
		\end{scriptsize}
	\end{tikzpicture}
	\caption{Complex blow-spherical at infinity of the curve $(y-x^2)x=0$.}
	\label{fig:complexbstreeinfinity_one}
  \end{minipage} 
  & 
  \begin{minipage}[c][6cm][c]{6cm}
	\centering
	\begin{tikzpicture}[line cap=round,line join=round,>=triangle 45,x=0.7cm,y=0.7cm]
		\clip(-5,-1) rectangle (5,6.2);
		\draw [line width=1pt] (0,0)-- (-1,2);
		\draw [line width=1pt] (0,0)-- (1,2);
		\draw [line width=1pt] (-1,2)-- (-1,4);
		\draw [line width=1pt] (1,2)-- (1,4);
		\draw [line width=1pt] (0,0) circle (2pt);
		\begin{scriptsize}
			\draw [fill=black] (0,0) circle (2pt);
			\draw [fill=black] (-1,4) circle (2pt);
			\draw [fill=black] (1,2) circle (2pt);
			\draw [fill=black] (-1,2) circle (2pt);
			\draw [fill=black] (1,4) circle (2pt);
			\draw [fill=black] (0.4 ,0) node {\normalsize $\infty$};
			\draw [fill=black] (-1.4 ,3) node {\normalsize $1$};
			\draw [fill=black] (1.4 ,3) node {\normalsize $1$};
		\end{scriptsize}
	\end{tikzpicture}
	\caption{complex blow-spherical tree at infinity of the curve $xy=0$.}
	\label{fig:complexbstreeinfinity_two}
\end{minipage}
\end{tabular}
\end{figure}

\bigskip

\begin{figure}[H]
\begin{tabular}{cc}
\begin{minipage}[c][6cm][c]{6cm}
	\centering
	\begin{tikzpicture}[line cap=round,line join=round,>=triangle 45,x=0.7cm,y=0.7cm]
		\clip(-5,-1) rectangle (5,6.2);
		\draw [line width=1pt] (0,0)-- (0,2);
		\draw [line width=1pt] (0,2)-- (0,4);
		\draw [line width=1pt] (0,0) circle (2pt);
		\begin{scriptsize}
			\draw [fill=black] (0,0) circle (2pt);
			\draw [fill=black] (0,4) circle (2pt);
			\draw [fill=black] (0,2) circle (2pt);
			\draw [fill=black] (0.4 ,0) node {\normalsize $\infty$};
			\draw [fill=black] (0.4 ,3) node {\normalsize $1$};
		\end{scriptsize}
	\end{tikzpicture}
	\caption{Real blow-spherical tree at infinity of the curve $y=0$.}
	\label{fig:complexbstreeinfinity_three}
  \end{minipage} 
  & 
  \begin{minipage}[c][6cm][c]{6cm}
	\centering
	\begin{tikzpicture}[line cap=round,line join=round,>=triangle 45,x=0.7cm,y=0.7cm]
		\clip(-5,-1) rectangle (5,6.2);
		\draw [line width=1pt] (0,0)-- (-1,2);
		\draw [line width=1pt] (0,0)-- (1,2);
		\draw [line width=1pt] (-1,2)-- (-1,4);
		\draw [line width=1pt] (1,2)-- (1,4);
		\draw [line width=1pt] (0,0) circle (2pt);
		\begin{scriptsize}
			\draw [fill=black] (0,0) circle (2pt);
			\draw [fill=black] (-1,4) circle (2pt);
			\draw [fill=black] (1,2) circle (2pt);
			\draw [fill=black] (-1,2) circle (2pt);
			\draw [fill=black] (1,4) circle (2pt);
			\draw [fill=black] (0.4 ,0) node {\normalsize $\infty$};
			\draw [fill=black] (-1.4 ,3) node {\normalsize $1$};
			\draw [fill=black] (1.4 ,3) node {\normalsize $1$};
		\end{scriptsize}
	\end{tikzpicture}
	\caption{Real blow-spherical tree at infinity of the curve $x^2+y^2=1$.}
	\label{fig:complexbstreeinfinity_four}
\end{minipage}
\end{tabular}
\end{figure}

\bigskip

\begin{theorem}\label{thmequivinfinity}
 Let $C$ and $\Gamma$ be two complex algebraic curves. Let $X_1,...,X_r$ (resp. $Y_1,...,Y_s$) be the irreducible components of $C(C,\infty)$ (resp. $C(\Gamma,\infty)$). We index the ends of $C$ and $\Gamma$ in such a way: 
 $$C\setminus B_R=\bigcup_{i=1}^r\bigcup_{j=1}^{e_i}{E_{ij}}$$
 and
 $$\Gamma\setminus B_R=\bigcup_{l=1}^s\bigcup_{m=1}^{\widetilde{e}_l}{F_{lm}},$$
 satisfying $C(E_{ij},\infty)=X_i$ for all $i\in\{1,\cdots,r\}$ and $C(F_{lm},\infty)=Y_l$ for all $l\in\{1,\cdots,s\}$. Then, the following statements are equivalent:
 \begin{enumerate}
  \item [(1)] $C$ and $\Gamma$ are blow-spherical homeomorphic at infinity;
  \item [(2)] there exists $\sigma\colon \{1,\cdots,r\}\rightarrow \{1,\cdots,s\}$ such that for each $i$, there exists $\sigma_i\colon \{1,\cdots,e_i\}\rightarrow \{1,\cdots, \widetilde{e}_{\sigma(i)}\}$ such that $k_{E_{ij},\infty}(X_i)=k_{F_{\sigma(i)\sigma_i(j)},\infty}(Y_{\sigma(i)})$;
  \item [(3)] The blow-spherical trees at infinity of $C$ and $\Gamma$ are isomorphic;
  \item [(4)] $C$ and $\Gamma$ are strongly blow-spherical homeomorphic at infinity.
 \end{enumerate}
\end{theorem}
\begin{proof}
 By identifying $\C^n$ with $\C^n\times \{0\}\subset \C^n\times \C^k=\C^{n+k}$, for any $k\in \mathbb{N}$, we may assume that $X$ and $Y$ are in the same $\C^n$.
 
$ (1)\Rightarrow (2)$. Let $h\colon C\setminus K\to \Gamma\setminus \tilde K$ be a blow-spherical homeomorphism, for some compact subsets $K$ and $\tilde K$. By Proposition \ref{coneinvarianteblow}, $h$ induces a homeomorphism $d_{\infty}h\colon C(C,\infty)\to C(\Gamma,\infty)$ and, consequently, $d_{\infty}h$ sends each irreducible component of $C(C,\infty)$ onto an irreducible component of $C(\Gamma,\infty)$. Thus, there is a bijection $\sigma\colon \{1,\cdots,r\}\rightarrow \{1,\cdots,s\}$ such that $Y_{\sigma(i)}=d_{\infty}h(X_i)$ for all $i\in \{1,\cdots,r\}$. For each $i\in \{1,\cdots,r\}$, by Proposition \ref{propinvarianciamultiplicidaderelativa}, $\{k_{E_{i,j},\infty}(X_i);j=1,...,e_i\}=\{k_{F_{\sigma(i),m},\infty}(Y_{\sigma(i)});m=1,...,\tilde e_{\sigma(i)}\}$. Then we can choose a bijection $\sigma_i\colon \{1,...,e_i\}\to \{1,...,\tilde e_{\sigma(i)}\}$ such that $k_{E_{i,j},\infty}(X_i)=k_{F_{\sigma(i),\sigma_i(j)},\infty}(Y_{\sigma(i)})$.

$ (2)\Rightarrow (3)$. We assume that there is a bijection $\sigma\colon \{1,\cdots,r\}\rightarrow \{1,\cdots,s\}$ such that for each $i\in \{1,\cdots,r\}$ there is a bijection $\sigma_i\colon \{1,\cdots,e_i\}\rightarrow \{1,\cdots,\widetilde{e}_{\sigma(i)}\}$ such that $k_{E_{ij},\infty}(X_i)=k_{F_{\sigma(i)\sigma_i(j)},\infty}(Y_{\sigma(i)})$. Then we define the isomorphism between the blow-spherical trees at infinity of $C$ and $\Gamma$ by sending the root on the root and sending each vertex $D_i$ (corresponding to the line $X_i$) to the vertex $\widetilde{D}_{\sigma(i)}$ (corresponding to the line $Y_{\sigma(i)}$), sending the edge joining the root and $D_i$ to the edge joining the root and $\widetilde{D}$, sending each vertex $A_{i,j}$ (corresponding to the end $E_{i,j}$) to $\widetilde{A}_{\sigma(i),\sigma_i(j)}$ (corresponding to the end $F_{\sigma(i),\sigma_i(j)}$) and sending the edge joining $A_{i,j}$ to $D_i$ to the edge joining $\widetilde{A}_{\sigma(i),\sigma_i(j)}$ and $\widetilde{D}$. 

$ (3)\Rightarrow (4)$. Let $\Psi$ be an isomorphism between the complex blow-spherical trees at infinity of $C$ and $\Gamma$. So, we have $\Psi(D_i)=\widetilde{D}_{\sigma(i)}$ and $\Psi(A_{i,j})=\widetilde{A}_{\sigma(i),\sigma_i(j)}$ such that $k_{E_{ij},\infty}(X_i)=k_{F_{\sigma(i)\sigma_i(j)},\infty}(Y_{\sigma(i)})$. 

Fix $i\in \{1,\cdots,r\}$. Let $T_{i}\colon \mathbb{C}^n\rightarrow \mathbb{C}^n$ and $L_{\sigma(i)}\colon \mathbb{C}^n\rightarrow \mathbb{C}^n$ be linear isomorphisms such that $T(X_i)=L_{\sigma(i)}(Y_{\sigma(i)})=L:=\{(0,\xi)\in \C^{n-1}\times \C;\xi \in \mathbb{C}\}$. Fix $j\in \{1,\cdots,e_i\}$. Let $E=T_i(E_{ij})$ and $F=L_{\sigma(i)}(F_{\sigma(i)\sigma_i(j)})$. Note that $T_i\colon E_{ij}\to E$ and $L_{\sigma(i)}\colon F_{\sigma(i)\sigma_i(j)}\to F$ are strong blow-spherical homeomorphisms at infinity. Thus, by Proposition \ref{propinvarianciamultiplicidaderelativa}, $k_{E_{ij},\infty}(X_i)=k_{E,\infty}(L)$ and $k_{F_{\sigma(i)\sigma_i(j)},\infty}(Y_{\sigma(i)})=k_{F,\infty}(L)$.

\begin{claim} \label{claim:propequivblowdeg}
$E$ and $F$ are strongly blow-spherical homeomorphic at infinity
\end{claim}
\begin{proof} 
Let $\psi\colon \mathbb{C}\setminus D_{\epsilon}\rightarrow E$ and $\widetilde{\psi}\colon\mathbb{C}\setminus D_{\epsilon}\rightarrow F$ be Puiseux's parametrizations at infinity, for some disc $D_{\epsilon}\subset \mathbb{C}$. Since $k=k_{E_{ij},\infty}(L)=k_{F_{\sigma(i)\sigma_i(j)},\infty}(L)$, there are holomorphic mappings $\phi,\widetilde{\phi}\colon\mathbb{C}\setminus D_{\epsilon}\to \C^{n-1}$ satisfying $\lim\limits_{t\to \infty} \frac{\phi(t)}{t^k}=\lim\limits_{t\to \infty} \frac{\widetilde{\phi}(t)}{t^k}=0$
and such that $\psi$ and $\widetilde{\psi}$ are homeomorphisms given by
 $$\psi(t)=(\phi(t),t^{k})$$
and 
 $$\widetilde{\psi}(t)=(\widetilde{\phi}(t),t^{k}).
 $$
 
 Now define $\varphi\colon E\rightarrow F$ given by $\varphi=\widetilde{\psi}\circ \psi^{-1}$. Note that $\varphi$ is a diffeomorphism. Thus, we have to prove that $\varphi$ is a blow-spherical homeomorphism at infinity. Indeed, let $(x,0)\in \partial C'_{\infty}$ and for any sequence $\{z_n\}_{n\in \mathbb{N}}\subset E'\setminus \partial {E'}_{\infty}$ such that $\lim_{n\rightarrow +\infty}{z_n}=(x,0)$ and for all $n$ we write $z_n=(x_n,t_n)$ and $s_n=\psi^{-1}(\frac{x_n}{t_n})$. Then, we have
 \begin{eqnarray*}
  \varphi'_{\infty}(z_n) & = & \left(\frac{\varphi(\frac{x_n}{t_n})}{\|\varphi(\frac{x_n}{t_n})\|},\frac{1}{\|\varphi(\frac{x_n}{t_n})\|}\right)\\
  & = & \left(\frac{(\widetilde{\phi}(s_n),s_n^{k})}{\|(\widetilde{\phi}(s_n),s_n^{k})\|},\frac{1}{\|(\widetilde{\phi}(s_n),s_n^{k})\|}\right).
 \end{eqnarray*}
On the other hand, $\frac{x_n}{t_n}=(\phi(s_n),s_n^{k})$ and consequently we have
$$z_n=\left(\frac{(\phi(s_n),s_n^{k})}{\|(\phi(s_n),s_n^{k})\|},\frac{1}{\|(\phi(s_n),s_n^{k})\|}\right).$$
Finally, since $\lim\limits_{t\to \infty} \frac{\phi(t)}{t^k}=\lim\limits_{t\to \infty} \frac{\widetilde{\phi}(t)}{t^k}=0$, we have 
$$\lim_{n\rightarrow +\infty}{\varphi'(z_n)}=\lim_{n\rightarrow +\infty}{z_n}=(x,0).$$
Therefore, $\varphi$ is a blow-spherical homeomorphism at infinity such that $\varphi'|_{\partial {E'}_{\infty}}=\mbox{id}_{\partial {E'}_{\infty}}$.
\end{proof}

Since $\varphi\colon E\to F$ is a strong blow-spherical homeomorphism at infinity, then $\varphi_{ij}:=L_{\sigma(i)}\circ \varphi\circ T_{i}^{-1}\colon E_{ij}\to F_{\sigma(i)\sigma_i(j)}$ is a strong blow-spherical homeomorphism at infinity. 

Now, we define $h\colon C\setminus  B_R\to \Gamma \setminus B_R$ by $h(x)=\varphi_{ij}(x)$ if $x\in E_{ij}$. Since each $\varphi_{ij}$ is a strong blow-spherical homeomorphism and for each $i\in\{1,...,r\}$, $\varphi_{ij}'|_{\partial {E_{ij}'}_{\infty}}=\varphi_{il}'|_{\partial {E_{il}'}_{\infty}}$ for all $j,l\in\{1,...,e_i\}$, we obtain that $h$ is a strong blow-spherical homeomorphism.

$ (4)\Rightarrow (1)$. Trivial.
\end{proof}

\subsection{Normal forms for the classification at infinity}


Let $p_1\colon \mathbb Z_{>0}\times \mathcal P(\mathbb Z_{>0})\to \mathbb Z_{>0}$ and $p_2\colon \mathbb Z_{>0}\times \mathcal P(\mathbb Z_{>0})\to \mathcal P(\mathbb Z_{>0})$ be the canonical projections, where $\mathcal P(\mathbb Z_{>0})$ denotes the power set of $\mathbb Z_{>0}$ and $\mathbb Z_{>0}=\{n\in \mathbb Z; n>0\}$;
Let $\mathcal A$ the subset of $\mathbb \mathbb Z_{>0}\times \mathcal P(\mathbb Z_{>0})$ formed by the all finite and non-empty subsets $A$ satisfying the following:
\begin{itemize}
\item [i)] $p_1(A)=\{1,...,r\}$ for some $r\in \mathbb Z_{>0}$; 
\item [ii)] $p_2(p_1^{-1}(i ))=\{k_{i, 1},...,k_{i ,e_{i }}\}\subset \mathbb Z_{>0}$ and $k_{i, j}\leq k_{i, j+1}$ for all $j\in\{1,...,e_{i}-1\}$ and for all $i\in \{1,...,r\}$.
\item [iii)] $e_{i }\leq e_{i +1}$ for all $i \in \{1,...,r-1\}$.
\end{itemize}
For a set $A\in \mathcal A$ as above, we define the realization of $A$ to be the curve
$$
X_A= \{(x,y)\in \C^2; \displaystyle\prod \limits_{i=1}^r\displaystyle\prod \limits_{j=1}^{e_i}(j(y-i x)^{k_{i, j}-1}-(y+i x)^{k_{i, j}})=0\}.
$$
Thus, it follows from Theorem \ref{thmequivinfinity} and definition of $\mathcal A$, the following classification result:
\begin{theorem}\label{thm:normal_forms}
For each complex algebraic curve $X\subset \C^n$, there exists a unique set $A\in \mathcal A$ such that $X_A$ and $X$ are blow-spherical homeomorphic at infinity.
\end{theorem}

\begin{remark}\label{rem:plane_spacial_are_same}
Theorem \ref{thm:normal_forms} says in particular that any spacial algebraic curve is blow-spherical homeomorphic at infinity to a plane algebraic curve.
\end{remark}

\subsection{Global classification of complex algebraic curves}

\begin{definition}
Let $C\subset \mathbb{C}^n$ be a complex algebraic curve and let $p_1,\cdots , p_s$ be the singular points of $C$. 
 The {\bf complex blow-spherical tree of $C$} is the rooted tree with a root corresponding to the curve, with vertices $C_j's$ corresponding to the irreducible components of the curve $C$, we put edges joining each vertex $C_j$ equipped with a weight given by the Euler characteristic of the corresponding irreducible component. Fixed $j$, for each singular point $p_i$ of $C$ such that $p_i\in C_j$, we add a new vertex with the label $i$, we add also a new vertex with the label $\infty$ and we put edges joining each such a new vertex and $C_j$. For each $i\in\{1,..,s\}$, let $X_{i,1},...,X_{i,d_i}$ be the irreducible components of $C(C,p_i)$. Let $X_{\infty,1},...,X_{\infty,d_{\infty}}$ be the irreducible components of $C(C,\infty)$. We put a new vertex with the label $m$ for each tangent direction $X_{\infty,m}$ satisfying $X_{\infty,m}\subset C(C_j,\infty)$ and we put edges joining each vertex $m$ with the vertex $\infty$. If $p_i\in C_j$, then we put a new vertex with the label $m$ for each tangent direction $X_{i,m}$ satisfying $X_{i,m}\subset C(C_j,p_i)$ and we put edges joining each vertex $m$ with the vertex corresponding to $p_i$. Finally, for each irreducible component  $C_{j,p_i}^l$ (resp. for each end $E_{jl}$) of the germ $(C_j,p_i)$ (resp. of $C_j$) which is tangent to $X_{i,m}$ (resp. $X_{\infty,m}$), we put a new vertex and an edge joining this new vertex and $m$ equipped with the weight given by the relative multiplicity $k_{C_{j,q_l}^l,p_i}(X_{i,m})$ (resp. $k_{E_{jl},\infty}(X_{\infty,m})$).
\end{definition}

Some examples of complex blow-spherical trees are presented in Figures \ref{fig:complexbstree} and \ref{fig:complexbstree_two}.
 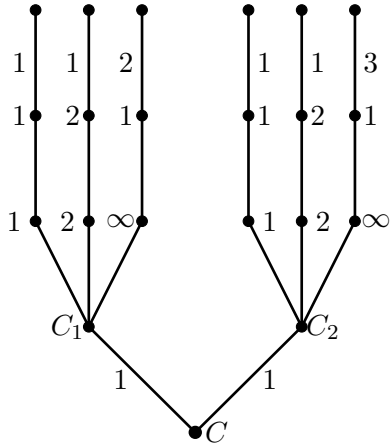
\begin{figure}[h]
	\centering
	\begin{tikzpicture}[line cap=round,line join=round,>=triangle 60,x=0.7cm,y=0.7cm]
		\clip(-5,-1) rectangle (5,8.2);
		\draw [line width=1pt] (0,0)-- (2,2);
		\draw [line width=1pt] (0,0)-- (-2,2);
		\draw [line width=1pt] (2,2)-- (1,4);
		\draw [line width=1pt] (2,2)-- (2,4);
		\draw [line width=1pt] (2,2)-- (3,4);
		\draw [line width=1pt] (-2,2)-- (-3,4);
		\draw [line width=1pt] (-2,2)-- (-2,4);
		\draw [line width=1pt] (-2,2)-- (-1,4);
		\draw [line width=1pt] (-1,4)-- (-1,6);
		\draw [line width=1pt] (-2,4)-- (-2,6);
		\draw [line width=1pt] (-3,4)-- (-3,6);
		\draw [line width=1pt] (1,4)-- (1,6);
		\draw [line width=1pt] (2,4)-- (2,6);
		\draw [line width=1pt] (3,4)-- (3,6);
		\draw [line width=1pt] (-1,6)-- (-1,8);
		\draw [line width=1pt] (-2,6)-- (-2,8);
		\draw [line width=1pt] (-3,6)-- (-3,8);
		
		\draw [line width=1pt] (1,6)-- (1,8);
		\draw [line width=1pt] (2,6)-- (2,8);
		\draw [line width=1pt] (3,6)-- (3,8);		
		
		\draw [line width=1pt] (0,0) circle (2pt);
		\begin{scriptsize}
			\draw [fill=black] (0,0) circle (2pt);
			\draw [fill=black] (2,2) circle (2pt);
			\draw [fill=black] (-2,2) circle (2pt);
			\draw [fill=black] (1,4) circle (2pt);
			\draw [fill=black] (2,4) circle (2pt);
			\draw [fill=black] (3,4) circle (2pt);
			\draw [fill=black] (-3,4) circle (2pt);
			\draw [fill=black] (-2,4) circle (2pt);
			\draw [fill=black] (-1,4) circle (2pt);
			\draw [fill=black] (-1,6) circle (2pt);
			\draw [fill=black] (-2,6) circle (2pt);
			\draw [fill=black] (-1,8) circle (2pt);
			\draw [fill=black] (-2,8) circle (2pt);
			\draw [fill=black] (-3,8) circle (2pt);
			\draw [fill=black] (1,8) circle (2pt);
			\draw [fill=black] (2,8) circle (2pt);
			\draw [fill=black] (3,8) circle (2pt);
			
			\draw [fill=black] (-3,6) circle (2pt);
			\draw [fill=black] (1,6) circle (2pt);
			\draw [fill=black] (2,6) circle (2pt);
			\draw [fill=black] (3,6) circle (2pt);
			\draw [fill=black] (0.4 ,0) node {\normalsize $C$};
			\draw [fill=black] (-1.4 ,1) node {\normalsize $1$};
			\draw [fill=black] (1.4 ,1) node {\normalsize $1$};
			\draw [fill=black] (-2.4 ,2) node {\normalsize $C_1$};
			\draw [fill=black] (2.4 ,2) node {\normalsize $C_2$};
			\draw [fill=black] (3.4 ,4) node {\normalsize $\infty$};
			\draw [fill=black] (2.4 ,4) node {\normalsize $2$};
			\draw [fill=black] (1.4 ,4) node {\normalsize $1$};
			\draw [fill=black] (-3.4 ,4) node {\normalsize $1$};
			\draw [fill=black] (-2.4 ,4) node {\normalsize $2$};
			\draw [fill=black] (-1.4 ,4) node {\normalsize $\infty$};
			\draw [fill=black] (-1.3 ,7) node {\normalsize $2$};
			\draw [fill=black] (-2.3 ,7) node {\normalsize $1$};
			\draw [fill=black] (-3.3 ,7) node {\normalsize $1$};
			\draw [fill=black] (1.3 ,7) node {\normalsize $1$};
			\draw [fill=black] (2.3 ,7) node {\normalsize $1$};
			\draw [fill=black] (3.3 ,7) node {\normalsize $3$};
			\draw [fill=black] (-1.3 ,6) node {\normalsize $1$};
			\draw [fill=black] (-2.3 ,6) node {\normalsize $2$};
			\draw [fill=black] (-3.3 ,6) node {\normalsize $1$};
			\draw [fill=black] (1.3 ,6) node {\normalsize $1$};
			\draw [fill=black] (2.3 ,6) node {\normalsize $2$};
			\draw [fill=black] (3.3 ,6) node {\normalsize $1$};
		\end{scriptsize}
	\end{tikzpicture}
	\caption{Complex blow-spherical tree of the curve $(y-x^2)(y-x^3)=0$.}
	\label{fig:complexbstree}
\end{figure}

 \begin{figure}[h]
\centering
	\begin{tikzpicture}[line cap=round,line join=round,>=triangle 60,x=0.7cm,y=0.7cm]
		\clip(-5,-1) rectangle (5,8.2);
		\draw [line width=1pt] (0,0)-- (0,2);
		\draw [line width=1pt] (0,2)-- (1,4);
		\draw [line width=1pt] (0,2)-- (-1,4);
		\draw [line width=1pt] (1,4)-- (1,6);
		\draw [line width=1pt] (-1,4)-- (-1,6);
		\draw [line width=1pt] (-1,6)-- (-1,8);
		\draw [line width=1pt] (1,6)-- (1,8);
			
		\draw [line width=1pt] (0,0) circle (2pt);
		\begin{scriptsize}
			\draw [fill=black] (0,0) circle (2pt);
			\draw [fill=black] (0,2) circle (2pt);
			\draw [fill=black] (1,4) circle (2pt);
			\draw [fill=black] (-1,6) circle (2pt);
			\draw [fill=black] (-1,8) circle (2pt);
			\draw [fill=black] (-1,4) circle (2pt);
			\draw [fill=black] (1,8) circle (2pt);
			
			\draw [fill=black] (1,6) circle (2pt);
			\draw [fill=black] (0.4 ,0) node {\normalsize $C$};
			\draw [fill=black] (0.4 ,2) node {\normalsize $C_1$};
			\draw [fill=black] (-1.4 ,4) node {\normalsize $1$};
			\draw [fill=black] (1.4 ,4) node {\normalsize $\infty$};
			\draw [fill=black] (1.4 ,6) node {\normalsize $1$};
			\draw [fill=black] (1.4 ,7) node {\normalsize $3$};
			\draw [fill=black] (-1.4 ,6) node {\normalsize $1$};
			\draw [fill=black] (-1.4 ,7) node {\normalsize $2$};
			\draw [fill=black] (0.4,1) node {\normalsize $1$};
			
		\end{scriptsize}
	\end{tikzpicture}
	\caption{Complex blow-spherical tree of the curve $y^2-x^3=0$.}
	\label{fig:complexbstree_two}
\end{figure}
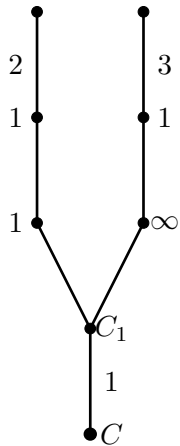

\begin{theorem}\label{main_theorem}
 Let $C=\bigcup_{i\in I}{C_i}$ and $\widetilde{C}=\bigcup_{j\in J}{\widetilde{C}_j}$ be two complex algebraic curves, where $\{{C_i}\}_{i\in I}$ and $\{\widetilde{C}_j\}_{j\in J}$ are the irreducible components of $C$ and $\widetilde{C}$, respectively. The following statements are equivalent:
 \begin{enumerate}
  \item [(1)] $C$ and $\widetilde{C}$ are blow-spherical homeomorphic;
  \item [(2)] There is an isomorphism between the complex blow-spherical trees of $C$ and $\widetilde C$;
  \item [(3)] $C$ and $\widetilde{C}$ are strongly blow-spherical homeomorphic.
 \end{enumerate}

\end{theorem}
\begin{proof}
$ (1)\Rightarrow (2)$. Assume that we have a blow-spherical homeomorphism $\varphi\colon C\rightarrow \widetilde{C}$. Thus we have that there is a bijection $\sigma\colon I\rightarrow J$ such that each irreducible component $C_i$ of $C$ is sent by $\varphi$ onto an irreducible component $\widetilde{C}_{\sigma(i)}$ of $\widetilde{C}$. Consequently  $\chi(C_i)=\chi(\widetilde{C}_{\sigma(i)})$, where $\chi(C)$ is the Euler characteristic of $C$. Since $\varphi({\rm Sing}_1(C))={\rm Sing}_1(\widetilde C)$, by Corollary 6.2 in \cite{Sampaio:2020} and by Theorem \ref{thmequivinfinity}, we can obtain an isomorphism between the complex blow-spherical trees of $C$ and $\widetilde C$.

$ (2)\Rightarrow (3)$. Assume that there is an isomorphism between the complex blow-spherical trees of $C$ and $\widetilde C$. Thus, we may assume that $C$ and $\widetilde C$ has the same complex blow-spherical tree.

We fix $r\in \{1,...,s\}$. We denote $X=C_r$ and $\widetilde X=\widetilde C_r$. Let $p_1,\cdots,p_s$ (resp. $\widetilde{p}_1,\cdots, \widetilde{p}_s$) be the singular points of $C$ (resp. $\widetilde{C}$) which are in $X$ (resp. $Y$). So, by hypothesis we have $\chi(X)=\chi(\widetilde{X})$, $k(X,\infty)=k(\widetilde{X},\infty)$ and $k(X,p_j)=k(\widetilde{X},\widetilde{p}_j)$, $j=1,\cdots, s$.

It follows from Corollary 6.2 in \cite{Sampaio:2020} that for a small enough $\epsilon>0$, there exist strong blow-spherical homeomorphisms
	$$g_j\colon C\cap \overline{B_{\epsilon}(p_j)}\rightarrow \widetilde{C}\cap \overline{B_{\epsilon}(\widetilde{p}_j)},$$
for any $j\in \{1,\cdots, s\}$.

By Theorem \ref{thmequivinfinity}, for a sufficiently large $R>0$, there exist a strong blow-spherical homeomorphism 
	$$h\colon X\setminus B_R(0)\rightarrow \widetilde{X}\setminus B_R(0).$$

We denote by $E_1,\cdots, E_{e}$ and $F_1,\cdots,F_e$ the ends of $X$ and $\widetilde{X}$, respectively, in such a way that $h(E_i\setminus B_R(0))= F_i\setminus B_R(0)$ for any $i\in \{1,\cdots, e\}$.

Now, we define the following surfaces with boundary
$$M= \Bigl(C\cap \overline{B_R(0)} \Bigr)\setminus \Bigl\{B_{\epsilon}(p_1)\cup \cdots \cup B_{\epsilon}(p_s)\Bigr\}$$
and 
$$N=\left(\widetilde{C}\cap \overline{B_R(0)}\right)\setminus \Bigl\{B_{\epsilon}(\widetilde{p}_1)\cup \cdots \cup B_{\epsilon}(\widetilde{p}_s)\Bigr\}.$$
Note that $f\colon \partial M\rightarrow \partial N$ is given by 
\[   
f(z) = 
\begin{cases}
	g_j(z), &\quad\text{if} \ z\in C; \|z-p_j\|=\epsilon\\
	h(z), &\quad\text{if} \ z\in X; \|z\|=R\\ 
\end{cases}
\]
 is a diffeomorphism.
 
The boundaries $\partial M$ and $\partial N$ are smooth compact manifolds of dimension 1 and, consequently, their connected components $S_i$ and $\widetilde{S}_i$ are diffeomorphic to $\mathbb{S}^1$, i.e., $\partial M=\bigcup_{i=1}^{k}{S_i}$ and $\partial N=\bigcup_{i=1}^{k}{\widetilde{S}_i}$.
 
 Since $M$ and $N$ are orientable and have the same genus and the same number of boundaries, we have a diffeomorphism $\nu\colon M\rightarrow N$ given by the theorem of classification for compact surfaces (see \cite{Hirsch}). The restriction $\nu\vert_{S_i}\colon S_i \rightarrow \widetilde{S}_i$ is an  orientation-preserving diffeomorphism. Now $\nu^{-1}\circ h_{i} \colon S_{i}\rightarrow S_i$ and by Theorem 3.3 in \cite{Hirsch} there is an isotopy $H_i\colon S_i\times [R',R]\rightarrow S_i\times [R',R]$ such that $H_i(\cdot,R')=\mbox{id}$ and $H_i(\cdot,R)=\nu^{-1}\circ h_i$. Let $s\colon [R',R]\rightarrow [R',R]$  a smooth function such that 
\[   
s(t) = 
\begin{cases}
	R', &\quad\text{if} \ t\in [R',\frac{R+2R'}{3}]\\
	g(t), &\quad\text{if} \ t\in [\frac{R+2R'}{3},\frac{2R-R'}{3}]\\ 
	R, &\quad\text{if} \ t\in[\frac{2R-R'}{3},R]\\ 
\end{cases},
\]
where $g$ is a smooth bump function such that 
\[   
g(t) = 
\begin{cases}
	R', &\quad\text{if} \ t\leq \frac{R+2R'}{3}\\
	R, &\quad\text{if} \ t\geq \frac{2R-R'}{3}\\ 
\end{cases}.
\]

Now we define $\Phi\colon C\cap \overline{ann(0;R',R)}\rightarrow \partial M\times [R',R]$ by $\Phi(x)=(h(x),|x|)$ and $\Psi\colon \partial N\times [R',R]\rightarrow \nu\left(C\cap \overline{ann(0;R',R)}\right)$ such that the diagram below commutes 

\begin{center}
\begin{tikzcd}
	C\cap \overline{ann(0;R',R)} \arrow[r]{d}{\Phi} \arrow[dr]{d}{\nu}
	& \partial M\times [R',R] \arrow[d]{d}{\Psi}\\
	& \nu\left(C\cap \overline{ann(0;R',R)}\right)
\end{tikzcd}
\end{center}
where $ann(0;R',R)=\overline{B_R(0)}\setminus B_{R'}(0)$.
So, we have 
\begin{eqnarray*}
	\Psi\left(H(h(x),R'),|x|\right) & = & \Psi\left(h(x),|x|\right)\\
	& = & \nu\circ \Phi^{-1}\left(\Phi(h(x))\right)\\
	& = &  \nu(h(x)),
\end{eqnarray*}
and 
\begin{eqnarray*}
	\Psi\left(H(h(x),R),|x|\right) & = & \Psi\left(\nu^{-1}\circ h(x),|x|\right)\\
	& = & \nu\circ \Phi^{-1}\left(\nu^{-1}\circ h(x),|x|\right)\\
	& = &  \nu\left(\nu^{-1}\circ h(x)\right)\\
	& = & h(x),
\end{eqnarray*}
By Theorem 3.3 in \cite{Hirsch}, there is a diffeomorphism $F\colon M\rightarrow N$ that extends $f\colon \partial M\rightarrow \partial N$. Therefore, we define the strong blow-spherical homeomorphism $\phi_r\colon X\rightarrow \widetilde{X}$ by:  

\[   
\phi_r(z) = 
\begin{cases}
	g_j(z), &\quad\text{if} \ z\in X; \|z-p_j\|\leq \epsilon\\
	h(z), &\quad\text{if} \ z\in X; \|z\|\geq R\\ 
	F(z), &\quad\text{if} \ z\in M.\\ 
\end{cases}
\]

Finally, the mapping $\varphi\colon C \rightarrow  \widetilde{C}$, defined by $\varphi(z) = \phi_r(z)$ whenever $z\in C_r$, is a strong blow-spherical homeomorphism.

Since $ (3)\Rightarrow (1)$ by definition, this finishes the proof.
\end{proof}

As a consequence we obtain the following example:
\begin{example}\label{main_example}
Let $C_1=\{(x,y)\in \C^2; y-x^4=0\}$, $C_2=\{(x,y)\in \C^2; y-x^2-x^4=0\}$ and $C_3=\{(x,y)\in \C^2; y-x^3-x^4=0\}$. Then $C=C_1\cup C_2$ and $\widetilde C=C_1\cup C_3$ are blow-spherical homeomorphic (see Figure \ref{fig:complexbstree_three}). Since the coincidence between $C_1$ and $C_2$ at $0$ is 2 and the coincidence between $C_1$ and $C_3$ at $0$ is 3 (see the definition of coincidence between curves in \cite{Fernandes:2003}), we have that $C$ and $\widetilde C$ are not outer lipeomorphic (see \cite[Theorem 3.2]{Fernandes:2003}).
\end{example}

\begin{figure}[H]
\begin{tabular}{cc}
\begin{minipage}[c][8cm][c]{6cm}
	\centering
	\begin{tikzpicture}[line cap=round,line join=round,>=triangle 45,x=0.7cm,y=0.7cm]
		\clip(-5,-1) rectangle (5,6.2);
		\draw [line width=1pt] (0,0)-- (2,2);
		\draw [line width=1pt] (0,0)-- (-2,2);
		\draw [line width=1pt] (-2,2)-- (-3,4);
		\draw [line width=1pt] (-2,2)-- (-1,4);
		\draw [line width=1pt] (2,2)-- (3,4);
		\draw [line width=1pt] (2,2)-- (1,4);
		\draw [line width=1pt] (1,4)-- (1,6);
		\draw [line width=1pt] (3,4)-- (3,6);
		\draw [line width=1pt] (-1,4)-- (-1,6);
		\draw [line width=1pt] (-3,4)-- (-3,6);
		\draw [line width=1pt] (0,0) circle (2pt);
		\begin{scriptsize}
			\draw [fill=black] (0,0) circle (2pt);
			\draw [fill=black] (2,2) circle (2pt);
			\draw [fill=black] (-2,2) circle (2pt);
			\draw [fill=black] (-3,4) circle (2pt);
			\draw [fill=black] (-1,4) circle (2pt);
			\draw [fill=black] (3,4) circle (2pt);
			\draw [fill=black] (1,4) circle (2pt);
			\draw [fill=black] (-1,6) circle (2pt);
			\draw [fill=black] (-3,6) circle (2pt);
			\draw [fill=black] (1,6) circle (2pt);
			\draw [fill=black] (3,6) circle (2pt);
			\draw [fill=black] (0.4 ,0) node {\normalsize $C$};
					\draw [fill=black] (-1.4 ,1) node {\normalsize $1$};
					\draw [fill=black] (1.4 ,1) node {\normalsize $1$};
			\draw [fill=black] (-2.6 ,2) node {\normalsize $C_1$};
			\draw [fill=black] (2.6 ,2) node {\normalsize $C_2$};
			\draw [fill=black] (-3.4 ,4) node {\normalsize $1$};
			\draw [fill=black] (-1.4 ,4) node {\normalsize $\infty$};
			\draw [fill=black] (1.4 ,4) node {\normalsize $1$};
			\draw [fill=black] (3.4 ,4) node {\normalsize $\infty$};
			\draw [fill=black] (-1.4 ,5) node {\normalsize $4$};
			\draw [fill=black] (-3.4 ,5) node {\normalsize $1$};
			\draw [fill=black] (1.4 ,5) node {\normalsize $1$};
			\draw [fill=black] (3.4 ,5) node {\normalsize $4$};
		\end{scriptsize}
	\end{tikzpicture}
	\caption{Complex blow-spherical tree of the curve $C$.}
	\label{fig:complexbstree_three}
  \end{minipage} 
  & 
  \begin{minipage}[c][8cm][c]{6cm}
	\centering
	\begin{tikzpicture}[line cap=round,line join=round,>=triangle 45,x=0.7cm,y=0.7cm]
		\clip(-5,-1) rectangle (5,6.2);
		\draw [line width=1pt] (0,0)-- (2,2);
		\draw [line width=1pt] (0,0)-- (-2,2);
		\draw [line width=1pt] (-2,2)-- (-3,4);
		\draw [line width=1pt] (-2,2)-- (-1,4);
		\draw [line width=1pt] (2,2)-- (3,4);
		\draw [line width=1pt] (2,2)-- (1,4);
		\draw [line width=1pt] (1,4)-- (1,6);
		\draw [line width=1pt] (3,4)-- (3,6);
		\draw [line width=1pt] (-1,4)-- (-1,6);
		\draw [line width=1pt] (-3,4)-- (-3,6);
		\draw [line width=1pt] (0,0) circle (2pt);
		\begin{scriptsize}
			\draw [fill=black] (0,0) circle (2pt);
			\draw [fill=black] (2,2) circle (2pt);
			\draw [fill=black] (-2,2) circle (2pt);
			\draw [fill=black] (-3,4) circle (2pt);
			\draw [fill=black] (-1,4) circle (2pt);
			\draw [fill=black] (3,4) circle (2pt);
			\draw [fill=black] (1,4) circle (2pt);
			\draw [fill=black] (-1,6) circle (2pt);
			\draw [fill=black] (-3,6) circle (2pt);
			\draw [fill=black] (1,6) circle (2pt);
			\draw [fill=black] (3,6) circle (2pt);
			\draw [fill=black] (0.4 ,0) node {\normalsize $\widetilde{C}$};
			\draw [fill=black] (-1.4 ,1) node {\normalsize $1$};
			\draw [fill=black] (1.4 ,1) node {\normalsize $1$};
			\draw [fill=black] (-2.6 ,2) node {\normalsize $\widetilde{C}_1$};
			\draw [fill=black] (2.6 ,2) node {\normalsize $\widetilde{C}_2$};
			\draw [fill=black] (-3.4 ,4) node {\normalsize $1$};
			\draw [fill=black] (-1.4 ,4) node {\normalsize $\infty$};
			\draw [fill=black] (1.4 ,4) node {\normalsize $1$};
			\draw [fill=black] (3.4 ,4) node {\normalsize $\infty$};
			\draw [fill=black] (-1.4 ,5) node {\normalsize $4$};
			\draw [fill=black] (-3.4 ,5) node {\normalsize $1$};
			\draw [fill=black] (1.4 ,5) node {\normalsize $1$};
			\draw [fill=black] (3.4 ,5) node {\normalsize $4$};
		\end{scriptsize}
	\end{tikzpicture}
	\caption{Complex blow-spherical tree of the curve $\widetilde{C}$.}
	\label{fig:complexbstree_four}

\end{minipage}
\end{tabular}
\end{figure}

\begin{remark}\label{rem:plane_spacial_are_not-same}
It was presented in \cite[Example 4.10]{BirbrairFJ:2021} an example of an algebraic curve $\Gamma\subset \C^n$ with nine cusps and such that its projective closure is a rational curve with nine cusps. Thus, there is no algebraic curve $\Lambda \subset \C^2$ which is blow-spherical homeomorphic to $\Gamma$. Indeed, assume that there is a blow-spherical homeomorphism $\varphi \colon \Gamma\to \Lambda$ and $\Lambda\subset \C^2$ is an algebraic curve.
By \cite[Theorem 4.3]{Sampaio:2020} the curve $\Lambda$ has also nine cusps, hence by Proposition \ref{homeoprojective} the projective closure of $\Lambda$ in $\C \mathbb{P}^2$ is a rational cuspidal curve with at least nine cusps, which is a contradiction by \cite[Corollary 1.2]{Tono:2005}. The above $\Gamma\subset \C^n$ was obtained as some affine part of a projective curve $\overline{\Gamma}\subset \mathbb{CP}^n$. Let us explain a way to obtain such a curve $\overline{\Gamma}$. Let $k$ be a positive integer number. We define $\varphi_k\colon \C\setminus \{1,...,k\}\to (\C^2)^k$ by 
$$
\varphi(t)=\textstyle{\left(t-1,\frac{1}{(t-1)^2},t-2,\frac{1}{(t-2)^2},...,t-k,\frac{1}{(t-k)^2}\right)}.
$$
Let $\overline{\Gamma}_k$ be the projective closure of the image of $\varphi$. We have that $\overline{\Gamma}_k$ is a rational algebraic curve with $k$ cusps.
\end{remark}

\section{An application to Lipschitz geometry}\label{sec:app_lne}
In this section, we show that blow-spherical homeomorphisms detect the complex algebraic curves which are LNE.

\begin{definition}[See \cite{BirbrairM:2000}]\label{def:lne}
Let $X\subset\R^n$ be a subset. We say that $X$ is {\bf Lipschitz normally embedded (LNE)} if there exists a constant $c\geq 1$ such that $d_{X,inn}(x_1,x_2)\leq C\|x_1-x_2\|$, for all pair of points $x_1,x_2\in X$.  We say that $X$ is {\bf Lipschitz normally embedded set at $p$} (shortly LNE at $p$), if there is a neighbourhood $U$ such that $p\in U$ and $X\cap U$ is an LNE set or, equivalently, that the germ $(X,p)$ is LNE. In this case, we say also that $X$ is $C$-LNE (resp. $C$-LNE at $p$). We say that $X$ is {\bf Lipschitz normally embedded set at infinity} (shortly LNE at infinity), if there is compact subset $K$ such that $X\setminus K$ is an LNE set. In this case, we say also that $X$ is $C$-LNE at infinity. 
\end{definition}

\begin{definition}\label{lipschitz function-infinity}
Let $X\subset\R^n$ and $Y\subset\R^m$. We say that $X$ and $Y$ are {\bf outer lipeomorphic at infinity} if there are compact subsets $K\subset \R^n$ and $\tilde K\subset \R^m$ and an outer lipeomorphism $f\colon X\setminus K\rightarrow Y\setminus \tilde K$.
\end{definition}

Let us remind the following result of Dias and Ribeiro in \cite{DiasR:2022}:
\begin{theorem}[Theorem 1.4 in \cite{DiasR:2022}]
 Let $X \subset \C^n$ be a pure dimensional complex algebraic subset. If $X$ is Lipschitz normally embedded at infinity, then $\deg(X)=\deg(C(X,\infty))$.
\end{theorem}
As a consequence, they also obtained that if a pure dimensional complex algebraic subset is Lipschitz normally embedded at infinity, then its tangent cone at infinity is reduced (see \cite[Proposition 3.3]{DiasR:2022}). The local version of this consequence was proved in \cite{DenkowskiT:2019} and it was proved in a more general setting in \cite{FernandesS:2019}.

Another consequence is the following:
\begin{corollary}
 Let $X \subset \C^n$ be a complex algebraic curve of degree $d$. If $X$ is Lipschitz normally embedded at infinity, then $C(X, \infty)$ is the union of $d$  different complex lines passing through the origin. 
\end{corollary}
The reciprocal of the above corollary also holds. In order to see that, let us remind the following result, which is part of the classification of curves under outer lipeomorphisms at infinity in \cite{Targino:2022}, where we are denoting by $\overline{X}$ the projective closure of $X$:

 \begin{theorem}\label{atinfinity}
Let $C$ and $C'$ be two complex algebraic plane curves.
The following statements are equivalent:
 \begin{enumerate}
      \item\label{it1} $C$ and $C'$ are outer lipeomorphic at infinity;
  \item\label{it2} there is a bijection $\psi$ between the set of points at infinity of $C$ and the set of points at infinity of $C'$ such that $(\overline{C}\cup L_\infty,p)$ has the same embedded topological type as $(\overline{C'}\cup L_\infty,\psi(p))$.
  \end{enumerate}
\end{theorem}

Since any complex algebraic curve $X \subset \C^n$ is outer lipeomorphic at infinity to a plane complex algebraic curve $A \subset \C^2$ (see \cite[Theorem 2.1]{FernandesJ:2023}), it follows from Theorem \ref{atinfinity} that a (connected) complex algebraic curve $X$ of degree $d$ is LNE infinity if and only if $C(X, \infty)$ is the union of $d$  different complex lines passing through the origin. Since a complex algebraic curve $X$ is locally the union of smooth curves which are pairwise transverse at $p$ if and only if $C(X, p)$ is the union of $m(X,p)$ different complex lines passing through the origin, the following local result was already known: {\it a complex algebraic curve $X$ is LNE at $p$ if and only if $C(X, p)$ is the union of $m(X,p)$ different complex lines passing through the origin} (e.g., see \cite{DenkowskiT:2019}). This also is implied by the classification of germs of complex analytic curves under outer lipemorphisms which was finished by Neumann and Pichon in \cite{N-P}, with previous contributions of Pham and Teissier in \cite{P-T} and Fernandes in \cite{Fernandes:2003} (see also \cite{FernandesSS:2018}). Thus, as an easy consequence of the above results, we have the following characterization of LNE curves:
\begin{proposition}
Let $X \subset \C^n$ be a connected complex algebraic curve of degree $d$. Then we the following statements are equivalent:
\begin{itemize}
 \item [(1)] $X$ is LNE;
 \item [(2)] $X$ is LNE at $p$ for all $p\in {\rm Sing}(X)\cup \{$infinity$\}$;
 \item [(3)] For all $p\in {\rm Sing}(X)$, $C(X, p)$ is the union of $m(X,p)$ different complex lines passing through the origin and $C(X, \infty)$ is the union of $d$ different complex lines passing through the origin.
\end{itemize}
\end{proposition}
An equivalent characterization appeared in the Ph.D. thesis of da Costa \cite{Costa:2023}. Apparently, the characterization in \cite{Costa:2023} is independent of the results of Targino, since there is no citation to \cite{Targino:2022}.

Since blow-spherical homeomorphisms detect the item (3) in the above proposition, we have the following:

\begin{proposition}\label{prop:charac_bs_lne}
Let $X \subset \C^n$ be a complex algebraic curve. Then $X$ is LNE if and only if it is blow-spherical homeomorphic to a complex algebraic curve which is LNE.
\end{proposition}

Let us remind the following result proved in \cite{FernandesS:2022}:

\begin{theorem}[Corollary 5.6 in \cite{FernandesS:2022}]
Let $X,Y\subset \mathbb{C}^n$ be two connected complex algebraic curves. Then, the following statements are equivalent:
\begin{itemize}
 \item [(1)] $X$ and $Y$ are homeomorphic;
 \item [(2)] $X$ and $Y$ are inner lipeomorphic.
\end{itemize}
\end{theorem}

As a consequence, we obtain the following characterization:
\begin{corollary}\label{cor:bs_lipeo}
Let $X,Y\subset \mathbb{C}^n$ be two LNE complex algebraic curves. Then, the following statements are equivalent:
\begin{itemize}
 \item [(1)] $X$ and $Y$ are homeomorphic;
 \item [(2)] $X$ and $Y$ are outer lipeomorphic;
 \item [(3)] $X$ and $Y$ are blow-spherical homeomorphic.
\end{itemize}
\end{corollary}
Let us remark that $(1)\Leftrightarrow (2)$ is also a consequence of the global classification in \cite{Targino:2022}, since the proof of the following result is an easy adaptation of the proof of the global classification in \cite{Targino:2022}:
 
\begin{proposition}\label{global}
Let $C$ and $\Gamma$ be two irreducible complex algebraic curves in $\C^n$. Then, $C$ and $\Gamma$ are outer lipeomorphic if and only if we have the following:
   \begin{enumerate}
     \item $C$ and $\Gamma$ are homeomorphic;
     \item $C$ and $\Gamma$ outer lipeomorphic at infinity and 
     \item there is a bijection $\varphi\colon {\rm Sing}(X)\to {\rm Sing}(Y)$ such that the germs $(C,p)$ and $(\Gamma,\varphi(p))$ are outer lipeomorphic for all $p\in {\rm Sing}(X)$.
 \end{enumerate}
\end{proposition}

\noindent{\bf Acknowledgements}. The authors would like to thank Alexandre Fernandes for his interest in this research.

\end{document}